\documentclass[11pt]{article}
\usepackage{amssymb,amsfonts,amsmath,latexsym,epsf,tikz,url}
\usepackage{epsfig}
\usepackage[space]{grffile} % For spaces in paths
\usepackage{etoolbox} % For spaces in paths
\usepackage{float}
\usepackage{soul}
\usepackage[colorinlistoftodos]{todonotes}
\usepackage{mathrsfs}
\usepackage[colorlinks]{hyperref}
\usetikzlibrary{arrows}
\makeatletter % For spaces in paths
\patchcmd\Gread@eps{\@inputcheck#1 }{\@inputcheck"#1"\relax}{}{}
\makeatother

\newtheorem{theorem}{Theorem}[section]

\newtheorem{proposition}[theorem]{Proposition}
\newtheorem{observation}[theorem]{Observation}
\newtheorem{corollary}[theorem]{Corollary}
\newtheorem{lemma}[theorem]{Lemma}

\newtheorem{definition}[theorem]{Definition}

\newcommand{\qed}{\hfill $\square$\medskip}

\begin{document}

\def\nt{\noindent}

\title{{\bf Strong coalitions in graphs}}

\bigskip 
\author{\small
H. Golmohammadi$^{1,2}$,  
S. Alikhani$^{3}$,
 N. Ghanbari$^{3}$,
I.I. Takhonov$^{1}$,
A. Abaturov$^{1}$
}

%\date{\today}

\maketitle

\begin{center}

$^{1}$Novosibirsk State University, Pirogova str. 2, Novosibirsk, 630090, Russia\\ 

\medskip
$^{2}$Sobolev Institute of Mathematics, Ak. Koptyug av. 4, Novosibirsk,
630090, Russia\\

\medskip
$^3$Department of Mathematical Sciences, Yazd University, 89195-741, Yazd, Iran\\

	\medskip
	{\tt h.golmohammadi@g.nsu.ru ~~ alikhani@yazd.ac.ir ~~ n.ghanbari.math@gmail.com ~~i.takhonov@g.nsu.ru ~~a.abaturov@g.nsu.ru }

\end{center}

%%%%%%%%%%%%%%ABSTRACT%%%%%%%%%%%%%%%%%%%%%%%%%%%%%%%%%%%%%%%%%%%%%%%%%%%%%%%%%%%%
\begin{abstract}
   For a graph $G=(V,E)$,  a set $D\subset V(G)$ is a strong dominating set of $G$, if for every vertex $x\in V (G)\setminus D$
   there is a vertex $y\in D$ with $xy \in E(G)$ and $deg(x)\leq deg(y)$.  A strong coalition consists of two disjoint sets of vertices $V_{1}$ and $V_{2}$, neither of which is a strong dominating set but whose union $V_{1}\cup V_{2}$, is a strong dominating set. A vertex partition $\Omega=\{V_1, V_2,..., V_k \}$ of vertices in $G$ is a strong coalition partition, if every set $V_i \in\Omega$ either
is a strong dominating set consisting of a single vertex of degree $n-1$, or is not a strong dominating set but produces a strong coalition with another set $V_j \in \Omega$ that is not a strong dominating set. The maximum cardinality of a strong coalition partition of $G$ is the strong
coalition number of $G$ and is denoted by $SC(G)$. In this paper, we study properties of strong coalitions in graphs.
\end{abstract}


\medskip
\noindent{\bf AMS Subj.\ Class.:}  05C60.

%%%%%%%%%%%%%%%%%%%%%%%%%%%%%%%%%%%%%%%%%%%%%%%%%%%%%%%%%%%%%%%%%%%%%%%%%%%%%%%%%
%%%%%%%%%%%%%%%%%%%%%%%%%%%%%%%%%%%%%%%%%%%%%%%%%%%%%%%%%%%%%%%%%%%%%%%%%%%%%%%%%

\section{Introduction}

 We begin with some basic definitions. Let $G$ be a
graph with vertex set $V=V(G)$, edge set $E=E(G)$, and order $n=|V|$. The minimum and
maximum degrees in $G$ are denoted by $\delta(G)$ and $\Delta(G)$, respectively. A universal vertex in a graph $G$ is a vertex adjacent to all other vertices, and so $deg(v)=n-1$. Domination is an extensively studied classic topic in graph theory and the literature on this subject has been thoroughly reviewed
in the books \cite{14,15}. Given a graph $G=(V,E)$, recall that a dominating set $S$ of a graph $G$ is a set such that every vertex in $V \setminus S$ is adjacent to some vertex in $S$. The domination number $\gamma(G)$ is the minimum cardinality of a dominating set. A set $D\subset V$ is a strong dominating set of $G$, if for every vertex $x\in V\setminus D$
there is a vertex $y\in D$ with $xy \in E(G)$ and $deg(x)\leq deg(y)$. The strong domination number $\gamma_{st}(G)$ is defined as the minimum cardinality of a strong dominating
set. The strong domination number was introduced in \cite{17} inspired by the real-life application to network traffic, and has subsequently been studied in \cite{1,5,8,16,18}. A domatic partition (or strong partition) is a partition
of the vertex set into dominating sets (or strong dominating sets). The maximum cardinality of a domatic partition (or strong partition) is called the domatic number (or strong domatic number), denoted by $d(G)$ (or $d_{st}(G)$). The domatic number of a graph was introduced in 1977 by  Cockayne and  Hedetniemi \cite{7}. Very recently, the strong domatic number of a graph has
started to be explored in \cite{9}. For more details on the domatic number refer to e.g., \cite{19,20}. 

 A coalition in a graph $G=(V,E)$ consists of two disjoint sets $V_1$ and $V_2$ of vertices, such that neither $V_1$ nor $V_2$ is a dominating set, but the union $V_1\cup V_2$
is a dominating set of $G$. A coalition partition in a graph $G$ of order $n=|V |$ is
a vertex partition $\mathcal{P}=\{V_1, V_2, . . . , V_k\}$ such that every set $V_i$ either is a dominating
set consisting of a single vertex of degree $n-1$, or is not a dominating set but
forms a coalition with another set $V_j$ which is not a dominating set. The maximum cardinality of a coalition partition
is the coalition number of a graph $G$, and is denoted by $C(G)$. In 2020, Hedetniemi et al. \cite{10} introduced and studied the concept of
coalition in graphs and have been
investigated in \cite{6,11,12,13}. For some recent papers on some variants of
coalitions in graphs see \cite{2,3,4}. In order to expand the study of coalitions, we  consider coalitions involving strong dominating sets in graphs. 

\medskip 
 This paper is organized as follows. In Section 2, we investigate the existence of strong
 coalition partitions. In Section 3, we present some bounds on the strong coalition number. In Section 4, we give a sufficient condition for graphs $G$ of order at least 4 with $SC(G)=|V(G)|$. In Section 5, we obtain the strong coalition
number of some graphs. Finally, we conclude the paper in Section 6. 
  \section{ Existence of strong coalition partitions}
  
We start this section with two main definitions.

 \begin{definition}[Strong coalition]
Two sets $V_1,V_2\subseteq V(G)$ form a strong coalition in a graph $G$ if they are not strong dominating sets but their union is a strong dominating set in $G$. 
 \end{definition} 
 
\begin{definition}[Strong coalition partition]\label{2.2} 
A strong coalition partition of a graph $G=(V,E)$ is a partition $\Omega=\{V_1,V_2,\ldots,V_k\}$ of the vertex set $V$ such that any $V_i\in \Omega, 1\leq i \leq k,$ is either a strong dominating set comprising a single vertex of degree $n-1$, or is not a strong dominating set but forms a strong coalition with another set $V_j \in \Omega$ that is not a strong dominating set. The maximum number of sets in $\Omega$ gives the strong coalition number $SC(G)$ of $G$. 
\end{definition}

The main goal of this section is to find out a necessary  condition for the existence of a strong coalition partition in a graph $G$. To achieve this goal, we shall now construct families of graphs have no strong coalition partitions. Let the family of graphs ${\cal F}$ contains all graphs $G\neq K_n$ of order $n$ with at least one vertex of degree $n-1$. The next result provides a necessary condition for the existence of a strong coalition partition.

\begin{proposition}~\label{1}
    Let $G$ be a graph. If $G\in {\cal F}$, then $SC(G)=0$.
\end{proposition}
\begin{proof} Suppose to contrary that $G$ has a $SC(G)$-partition $\Omega$. Since $G\in {\cal F}$, so $G$ contains at least one vertex of degree $n-1$ such as $x$. So it must be a singleton set of $\Omega$. Now, we may assume that each of $V_i \in\Omega$ and $V_j \in\Omega$ contains at least two vertices, while neither of $V_i$ nor $V_j$ is a singleton strong dominating set but $V_i \cup V_j$ can produce a strong coalition. It holds that there are vertices such as $u \in V_i$ and $v\in V_j$ such that $xu \in E(G)$ and $xv \in E(G)$. However, $deg(u)<deg(x)$ and $deg(v)<deg(x)$, a contradiction. Therefore  $G$ has no strong coalition partition and so  $SC(G)=0$.\qed

\end{proof}
    
As an immediate consequence of Proposition \ref{1}, we have the following observation.
 \begin{observation}
    
	\begin{enumerate} 
 
		\item[(i)] 
		If $G$ is a star graph $K_{1,n}$ where $n \ge 2$, then  $SC(G)=0$.
		
		\item[(ii)] If $F_n$ is a friendship graph, then $SC(G)=0$.
		 \end{enumerate} 
\end{observation}

\medskip

\section{Bounds on the strong coalition number}

In this section, we obtain some bounds for the strong coalition number. First we establish a relation between the strong coalition number $SC(G)$ and the strong domatic number $d_{st}(G)$ as follows.
\begin{theorem}\label{e1}
If $G$ is a graph of order $n$ contains no universal vertex, then $SC(G) \ge 2d_{st}(G)$.

\end{theorem} 

\begin{proof} 
	 Let $G$ has a strong domatic partition $\mathcal{S}=\{S_1, S_2,\ldots, S_k\}$ with $d_{st}(G)=k$. Since $G$ has no vertices of degree $n-1$ then $|S_i|>1$ for any $i$.  Without loss of generality we assume that the sets $\{S_1,S_2,\ldots,
S_{k-1}\}$ are minimal strong dominating sets. Indeed, if for some $i$, the set $S_i$ is not minimal, we find a subset $S'_i\subseteq S_i$ that is a minimal strong dominating set, and add the remaining vertices to the set $S_k$. Note that if we partition a minimal strong dominating set with more than one element into two non-empty sets, we obtain two non-strong dominating sets that together form a strong coalition. As a result, we divide each non singleton set $S_i$ into two sets $S_{i,1}$ and $S_{i,2}$ that form a strong coalition. This gives us a new partition $\mathcal{S}'$ consisting of non-strong dominating sets that pair with some other non-strong dominating set in $\mathcal{S}'$ form a strong coalition. 

We now check the strong dominating set $S_k$. 

If $S_k$ is a minimal strong dominating set, we divide it into two non-strong dominating sets, add these sets to $\mathcal{S}'$, and obtain a strong coalition partition of order at least $2k$. Then, since $k=d_{st}(G)$, $SC(G)\ge 2d_{st}(G)$.

If $S_k$ is not a minimal strong dominating set, we aim to get a subset $S'_k\subseteq S_k$ that holds this condition. Again, we use the strategy on partitioning $S'_k$ into two non-strong dominating sets giving together a strong coalition. Afterwards, we define $S''_k$ as the complement of $S'_k$ in $S_k$, and append $S'_{k,1}$ and $S'_{k,2}$ to $\mathcal{S}'$. If $S''_k$ can merge
with any non-strong dominating set to form a strong coalition, one can obtain a strong coalition partition of a cardinality at least $2k+1$ by adding $S''_k$ to $\mathcal{S}'$. Then, $SC(G)\ge 2d_{st}(G)+1$. However, if $S''_k$ can not form a strong coalition with any set in $\mathcal{S}'$, we eliminate $S'_{k,2}$ from $\mathcal{S}'$ and add the set $S'_{k,2}\cup S''_k$ to $\mathcal{S}'$. This leads to a strong coalition partition of a cardinality at least $2k$. Then, $SC(G)\ge 2d_{st}(G)$.

Due to the above arguments, we  conclude that $SC(G)\ge 2d_{st}(G)$. \qed
 \end{proof}

Since for any graph $G$ we have $d_{st}(G)\geq 1$ then by Theorem~\ref{e1} we immediately have the following statements.
\begin{corollary}
If $G$ is a graph of order $n>1$ with no vertices of degree $n-1$, then $SC(G)\ge 2$.
\end{corollary}

\begin{corollary} \label{corlesstwo}
If $G$ is a graph with $SC(G)<2$, then $G$ contains at least one vertex of degree $n-1$.
\end{corollary}

Now, we characterize the graphs $G$ having $SC(G)=1$.
\begin{lemma}\label{l1}
 $SC(G)=1$ if and only if $G=K_1$.
\end{lemma}
\begin{proof} If $SC(G)=1$, then $\{V\}$ is a $SC(G)$-partition. By definition strong $c-$partition, we must
have $|V| = 1$ and so $G=K_1$. Conversely, if $G=K_1$, clearly we have
$SC(G)=1$. \qed
\end{proof}

By Proposition~\ref{1} and Lemma~\ref{l1} we obtain the following result.

\begin{corollary}~\label{cor1}
If $G\not\in\mathcal{F}$ is a graph of order $n$, then $1\leq SC(G) \leq n$.
\end{corollary}

The upper bound of Corollary~\ref{cor1} is achieved  by complete graphs $K_n$.\medskip

We next establish a key result, which gives us the number of strong coalitions involving
any set in a $SC(G)$-partition of $G$.

\begin{theorem}\label{thm:Delta1}
	Let $G$ be a graph with maximum degree $\Delta(G)$,
and let $\Omega$ be a $SC(G)$-partition. If $X \in \Omega$, then $X$ is in at most $\Delta(G)+1$ strong coalitions.
\end{theorem} 
\begin{proof} Let $\Omega=\{V_1, V_2, \ldots, V_k\}$ be a $SC(G)$-partition. Without loss of generality, let $V_1$ is a
strong dominating set, then it is a singleton set and is in no strong coalitions. Now, we may assume that $V_1$ be a non-singleton set. It follows by definition that $V_1$ is not a strong dominating set. Therefore there is a vertex not in $V_1$, call it $x$, which is not strongly dominated by $V_1$. Let $x \in V_2$. Now, let each of $V_3, V_4, \ldots, V_\ell$ forms a strong coalition with $V_1$, that is, $V_1 \cup V_3$, $V_1 \cup V_4$, \ldots, $V_1 \cup V_\ell$ are all strong dominating sets, where $\ell \le k$. It follows that since $V_1$ does not strongly dominate vertex $x$, that each set $V_3, V_4, \ldots, V_\ell$ must strongly dominate vertex $x$. This means, in particular, that each set $V_3, V_4, \ldots, V_\ell$ contains a vertex $y$ which is adjacent to vertex $x$, and $deg(y) \ge deg(x)$. Thus, $\ell-2 \le deg(x) \le \Delta(G)$. This implies that $V_1$ can form a strong coalition with at most $\Delta(G)$ sets in $V_3, V_4, \ldots, V_\ell$. This brings us to consider $V_1 \cup V_2$, which might be a strong dominating set. If it is a strong dominating set, then $V_1$ can form a strong coalition with $\Delta(G)+1$ sets (including set $V_2$), otherwise $V_1$ can form a strong coalition with at most $\Delta(G)$ sets. \qed

\end{proof}

    Now, we establish a sharp upper bound on the strong coalition in terms of the maximum degree of $G$.

\begin{theorem}\label{thm:Delta2}
	For any graph $G$ with $\delta(G)=1$, we have
	$SC(G) \leq 2+2\Delta(G).$
\end{theorem} 
\begin{proof} 
If $\Delta(G)\geq \frac{n-1}{2}$, then by Corollary \ref{cor1} we have the result. Now, suppose that $\Delta(G) < \frac{n-1}{2}$. Let $V(G)=\{v_1,v_2,\ldots,v_n\}$ and $\Omega =\{V_1,V_2,\ldots,V_t\}$ be a $SC(G)$-partition. Suppose that $\deg (v_1)=1$. Every set of $\Omega$ that does not contain either $v_1$ and $v_2$ must be in a strong coalition with
a set containing either $v_1$ and $v_2$. If $v_1$ and $v_2$ are in the same
set of $\Omega$, then since $\deg (v_1)=1$, Theorem \ref{thm:Delta1} implies that $SC(G) \leq 1+1+\Delta(G)$. So, clearly $SC(G) \leq 2+2\Delta(G)$.  Now, suppose that $v_1$ and $v_2$ are in different sets
of $\Omega$. Without loss of generality, let $v_1\in V_1$ and $v_2\in V_2$. So, each of $V_1$ and $V_2$ is in a strong coalition with at most $1+\Delta(G)$ sets. If $V_1$ and $V_2$ form a strong coalition, then
$SC(G) \leq 2+2\Delta(G)$ as each of $V_1$ and $V_2$ can be in at most $\Delta(G)$ additional
strong coalitions. Otherwise, if $V_1$ and $V_2$ do not form a strong coalition, we also have the result. To show it, suppose that $t>2+2\Delta(G)$. Without loss of generality, we may
assume that $V_1$ is in a strong coalition with each of $V_3$, $V_4$, $\ldots$, $V_{\Delta(G) +2}$ and $V_{\Delta(G) +3}$, while
$V_2$ is in a strong coalition with each of $V_{\Delta(G) +4}$, $V_{\Delta(G) +5}$, $\ldots$, $V_{2\Delta(G) +2}$ and $V_{2\Delta(G) +3}$. Let $v_i$ be a vertex that
is not strongly dominated by $V_1$ and  $\deg(v_i)\leq\Delta(G)$. Without loss of generality, suppose that $N(v_i)=\{v_1',v_2',\ldots,v_{\Delta(G)}'\}$, $v_i\in V_3$, and $v_1'\in V_4$, $v_2'\in V_5$, $\ldots$, $v_{\Delta(G)-1}'\in V_{\Delta(G) +2}$ and $v_{\Delta(G)}'\in V_{\Delta(G) +3}$. But now $V_2 \cup V_{\Delta(G) +4}$ does not strong dominate
$v_i$, contradicting that $V_2$ and  $V_{\Delta(G) +4}$ form a strong coalition in $G$.
\qed
\end{proof}

Applying Theorem \ref{thm:Delta2} to the class of trees, we have the following result.

\begin{corollary}
If $T$ is a tree of order $n$, then $SC(T) \leq 2+2\Delta(T)$.

\end{corollary}

We next state the following result.

\begin{proposition}\label{1-Delta}
	Let $G$ be a graph of order $n$ with maximum degree $\Delta\leq n-2$. If $G$ has only one vertex with maximum degree and $r$ strong dominating sets, then $SC(G)\leq r+1$.
\end{proposition} 

\begin{proof}
	Suppose that $u$ is the vertex with maximum degree. Since $u$ should be in every strong dominating sets like $S$, so $S\setminus \{u\}$ and $\{u\}$ form a strong coalition and we are done. 
	\qed
\end{proof}

\section{Graphs $G$ with $SC(G)=|V(G)|$}

\begin{figure}
\begin{center}
\includegraphics[width=0.6\linewidth]{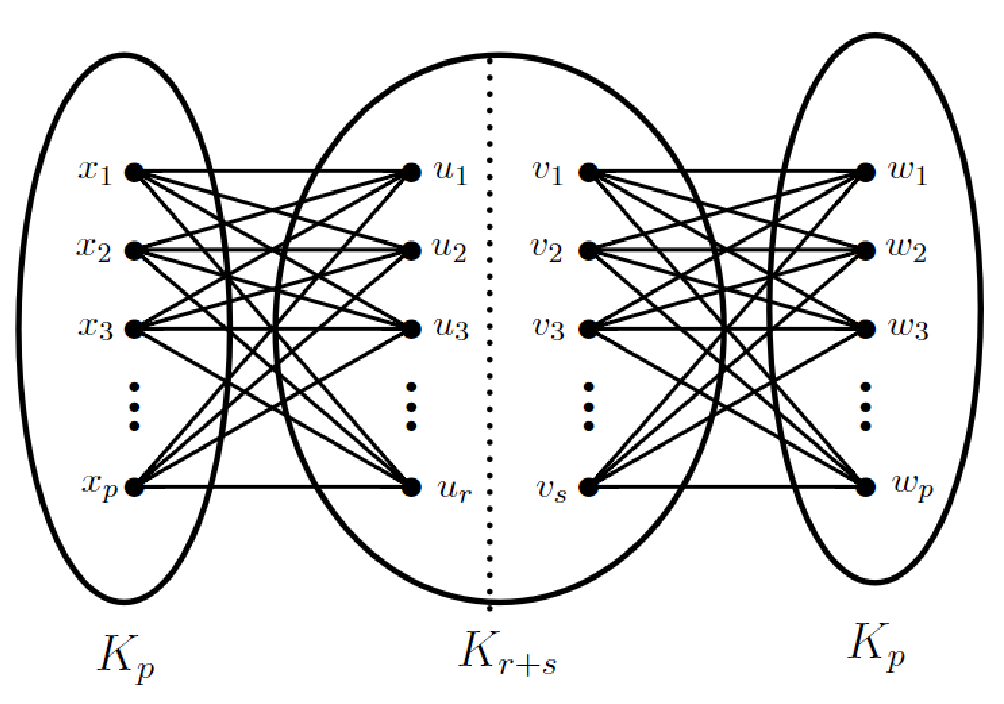}
\caption{\label{Characterize:n} Family ${\cal G}$}
\end{center}
\end{figure}

Studying graphs  $G$ with large and small strong coalition numbers form interesting problems. In this short section, we give a sufficient condition for graphs $G$ of order at least 4 with $SC(G)=|V(G)|$. To construct a family ${\cal G}$ of graphs with these properties, consider a complete  graph $K_{r+s}$, where its vertex set is $A\cup B$ such that $A=\{u_1,...,u_r\}$, $B=\{v_1, ..., v_s\}$, $A\cap B=\emptyset$,   $r\geq 0$, $s\geq 0$. 
Any graph $G$ of order at least 4  in the family ${\cal G}$ has (see Figure \ref{Characterize:n}), $V(G)=A\cup B\cup 2V(K_p)$  ($p\geq 0$), where 
$V(K_p)_1=\{x_1,x_2,...,x_p\}$ and $V(K_p)_2=\{w_1,w_2,...,w_p\}$ and 
\[
E(G)=E(K_{r+s})\cup E(2K_p)\cup \{x_iu_j: 1\leq i\leq p, 1\leq j\leq r\}\cup \{w_iv_j: 1\leq i\leq p, 1\leq j\leq s\}.
\]
 
\begin{theorem}
If  $G\in \cal G$ is a graph of order $|V(G)|\geq 4$, then $SC(G)=|V(G)|$.
\end{theorem} 

\begin{proof}
Suppose that $G\in \cal G$ and $V(G)=\{x_1,\ldots,x_p,w_1,\ldots,w_p,u_1,\ldots,u_r,v_1,\ldots,v_s,\}$ as we see in Figure \ref{Characterize:n}. We claim that the set consist of singleton sets is a $SC(G)$-partition of $G$. If $p=0$ then we have a complete graph and $SC(G)=|V(G)|$. If $r=0$ or $s=0$ or $r=s=0$, then we have a disjoint union of two complete graphs and $SC(G)=|V(G)|$. Now suppose that $r\geq 1$, $s\geq 1$ and $p\geq 1$. Then $\{x_i\}$ forms a strong coalition with $\{v_j\}$ for $1\leq i \leq p$ and $1\leq j \leq s$, $\{w_{i'}\}$ forms a strong coalition with $\{u_{j'}\}$ for $1\leq i' \leq p$ and $1\leq j' \leq r$, and also $\{u_{i''}\}$ forms a strong coalition with $\{v_{j''}\}$ for $1\leq i'' \leq r$ and $1\leq j'' \leq s$. Therefore, we have $SC(G)=|V(G)|$.

\qed
\end{proof}

\section{Strong coalition number for some classes of graphs }

In this section, we determine the strong coalition number of some classes of graphs.

\begin{proposition}\label{thm:scoalition-bipartite}
    
	For the complete bipartite graph $K_{r,s}$,
\begin{displaymath}
SC(K_{r,s})= \left\{ \begin{array}{ll}
2 & \textrm{~~~~ if ~~ $r>s>1$}\\
\\
2r & \textrm{~~~~ if ~~$r= s$}

\end{array} \right.
\end{displaymath}
\end{proposition}
\begin{proof}
Let $G= K_{r,s}$ with $r>s>1$. We need at least $s$ vertices to have a strong dominating set and all the vertices of the part with $s$ vertices should be in our set. Clearly a strong coalition partition can not have more than two sets and so $SC(K_{r,s})=2$. Now,  may assume that $r=s$. Since $\gamma_{st}=2$, no singleton set can be a strong dominating set. The partition $\Omega=\{V_1=\{v_1\},V_2=\{v_2\},...,V_{2r}=\{v_{2r}\}\}$ is a strong coalition partition of $K_{r,r}$ such that no set of $\pi$ is a strong dominating set but each set
$V_i$, for $1\leq i\leq 2r$, forms a strong coalition with another set $V_j$, for $1\leq j\leq 2r$. Hence, we have $SC(K_{r,r})=2r$.
\qed
\end{proof}

\begin{observation}{\rm \cite{7}}\label{thm:coalition-bipartite}
If $G$ is a complete bipartite graph $K_{r,s}$ such that $r>s>1$, then $C(K_{r,s})=r+s$.
\end{observation}

By Proposition \ref{thm:scoalition-bipartite} and Observation \ref{thm:coalition-bipartite}, we have the following corollary.
\begin{corollary}
For every integer $m$ there exists a graph $G$ such that $C(G)-SC(G)=m$.
\end{corollary}

Now, we determine the strong coalition numbers of all cycles. First, we recall the known result on the coalitions in cycles.

\begin{theorem}{\rm\cite{7}}\label{thm:coalition-cycle}
	For the cycle $C_n$,
\begin{displaymath}
C(C_n)= \left\{ \begin{array}{ll}
n & \textrm{~~~~ if ~~$n \leq 6$, }\\
\\
5 & \textrm{~~~~ if ~~$n = 7$,}\\
\\
6 & \textrm{~~~~ if ~~$n \geq 8$.}
\end{array} \right.
\end{displaymath}
\end{theorem} 

In order to give the strong coalition numbers of all cycles, we state the following lemma.

\begin{lemma}\label{lemma:regular}
If  $G$ is a regular graph, then $SC(G)=C(G)$.
\end{lemma} 

\begin{proof}
Since in a $k$-regular graph, the degree of all vertices are $k$, then a dominating set is a strong dominating set and vice versa. So by finding a coalition partition, we have a strong coalition partition and vice versa. Therefore we have the result. 
\qed
\end{proof}

Now, by using Theorem \ref{thm:coalition-cycle} and Lemma \ref{lemma:regular}, we have the following result.

\begin{corollary}
    
	For the cycle $C_n$,
\begin{displaymath}
SC(C_n)= \left\{ \begin{array}{ll}
n & \textrm{~~~~ if ~~$n \leq 6$ }\\
\\
5 & \textrm{~~~~ if ~~$n = 7$}\\
\\
6 & \textrm{~~~~ if ~~$n \geq 8$}
\end{array} \right.
\end{displaymath}
\end{corollary} 

Next we obtain the strong coalition numbers of paths.

\begin{theorem}
	For the path $P_n$,

\begin{displaymath}
SC(P_n)= \left\{ \begin{array}{ll}

1 & \textrm{~~~~ if ~~$n=1$}\\
\\
2 & \textrm{~~~~ if ~~$n=2$}\\
\\
0 & \textrm{~~~~ if ~~$n=3$}\\
\\
4 & \textrm{~~~~ if ~~$ 4 \leq n\leq 7$}\\
\\
5 & \textrm{~~~~ if ~~ $8 \leq n\leq 11$}\\
\\
6 & \textrm{~~~~ if ~~$n \geq 12$}
\end{array} \right.
\end{displaymath}

\end{theorem} 

\begin{proof}
Let $V(P_n)=\{v_1,v_2,\ldots,v_n\}$ such that $\deg(v_1)=\deg(v_n)=1$ and $\deg(v_i)=2$ for $2\leq i\leq n-1$ and $v_{i-1}$ and $v_i$ are adjacent, for $2\leq i\leq n$. By Lemma \ref{l1}, we have $SC(P_1)=1$. If $n=2$,
 then Proposition \ref{thm:scoalition-bipartite} gives the result. By Proposition \ref{1}, we have $SC(P_3)=0$. Now suppose that $n=4$. Since no single vertex strongly dominates $P_4$ but any pair of non-adjacent vertices strongly dominate $P_4$, the
partition containing each vertex in a singleton set is a
$SC(G)$-partition of $P_4 $ of order $4$. Now assume that $n=5$ and let $\Omega=\{\{v_1\},\{v_2\},\{v_3\},\{v_4\},\{v_5\}\}$ be a $SC(G)$-partition of $P_5$ into singleton sets. Since the graph $P_5$ has a unique  strong dominating set of order two, i.e., $\{v_2,v_4\}$, then each set containing vertices $v_1$, $v_3$, and $v_5$ is not in a strong coalition with any other set of $\Omega$. So only the sets that contain $v_2$ and $v_4$ are classified as singleton sets.
Thus, in order for
$\Omega$ to be a $SC(G)$-partition of $P_5$, $\Omega$ must contain at least one
non-singleton set. So, $SC(P_5)\leq 4$. The partition
 $\{V_1=\{v_3, v_5\},V_2=\{v_1\},V_3=\{v_2\},V_4=\{v_4\}\}$ is a $SC(G)$-partition of $P_5$, in which $V_1$ forms a strong coalition with each of $V_2$ and $V_3$, and $V_4$ forms a strong coalition with $V_3$. Now, we consider $n=6$. Let $\Omega=\{\{v_1\},\{v_2\},\{v_3\},\{v_4\},\{v_5\},\{v_6\}\}$ be a $SC(G)$-partition of $P_6$. Similar to  $P_5$, any strong coalition partition of $P_6$ must contain at
least one non-singleton set. Therefore, $SC(P_6)\leq 5$. Now, we show that $SC(P_6)\ne 5$. Let $\Omega=\{V_1,V_2,V_3,V_4,V_5\}$ be a $SC(G)$-partition of $P_6$. Suppose that $V_1=\{v_1\}$. It is clear that $V_1$ can not form a strong coalition with a singleton set. It follows that a strong coalition partner of $V_1$ must have at least two elements. One can easily check that the only option is $V_2=\{v_3,v_5\}$. If other sets be singleton sets, none of them can form a strong coalition with the singleton set consisting of $v_4$. The same argument holds for the singleton set consisting of $v_6$. So we need at least two sets with two vertices or one set with at least 3 vertices. Therefore in this case $SC(P_6)\ne 5$. Now assume that $|V_1|=2$ and $|V_i|=1$, for $2\leq i \leq 5$. Let $V_1=\{v_1,v_j\}$ such that $j\in\{2,3,4,5,6\}$. If $i\ne 6$, then the singleton set consisting of $v_6$ can only form a strong coalition with the doubleton set consisting of $\{v_2,v_4\}$. If $V_1=\{v_1,v_6\}$, then $V_1$ must be in a strong coalition with a set with size at least two. So in all cases, there is no $SC(G)$-partition of $P_6$ of order 5. One can easily check that the partition
 $\{V_1=\{v_1, v_6\},V_2=\{v_3,v_4\},V_3=\{v_2\},V_4=\{v_5\}\}$ is a $SC(G)$-partition of $P_6$.  Next assume that $n=7$. By Theorem \ref{thm:Delta2}, we have $SC(P_7)\leq 6$. We show that $SC(P_7)\ne 6$. Let $\Omega$ be a $SC(G)$-partition of $P_7$. Note that $\Omega$ consists of a set, say $V_1$, of cardinality $2$ and five singleton
sets. Since $\gamma_{st}(P_7)=3$, each singleton set must form a strong coalition with $V_1$, and by Theorem \ref{thm:Delta1}, $V_1$ has
at most three strong coalition partners, a contradiction. Therefore, $SC(P_7)\leq 5$. With similar method we see that  $SC(P_7)\ne 5$. The partition
 $\{V_1=\{v_1, v_3\},V_2=\{v_2, v_4\},V_3=\{v_6\}, V_4=\{v_5, v_7\}\}$ is a $SC(G)$-partition of $P_7$. Next suppose that $n=8$. By Theorem \ref{thm:Delta2}, we know that $SC(P_8)\leq 6$. We first show that $SC(P_8)\ne 6$. Let $\Omega$ be a $SC(G)$-partition of $P_8$. Now, we consider two cases.\medskip
 
 \nt {\bf Case 1.} Assume that $\Omega=\{V_1,V_2,V_3,V_4,V_5,V_6\}$  consists of a set of cardinality 3 and five singleton
sets. Since $\gamma_{st}(P_8)=3$, each singleton set must be a strong coalition of $V_1$, and by Theorem \ref{thm:Delta1}, $V_1$ has
at most three strong coalition partners, a contradiction. So, $SC(P_8)\leq 5$.

 \nt {\bf Case 2.} Let $|V_1|=|V_2|=2$ and $|V_i|=1$ for $3\leq i \leq 6$. Since $\gamma_{st}(P_8)=3$, each singleton set must form a strong coalition with a set of cardinality 2. Due to the graph $P_8$ has only two strong dominating sets of order 3, which consist of vertices $v_2$, $v_4$, $v_5$, and $v_7$, then each set containing vertices $v_1$, $v_3$,
$v_6$, and $v_8$ does not appear in any strong coalition of order 3. Without loss of generality, we may assume that $V_1=\{v_2,v_7\}$, $V_3=\{v_4\}$, and $V_4=\{v_5\}$ such that $V_1 \cup V_3$ and $V_1 \cup V_4$ are strong dominating sets of order 3. Therefore, each of $V_5$ and $V_6$ is not in a strong coalition with $V_1$ or $V_2$, a contradiction. Hence, $SC(P_8)\leq 5$.

The partition $\{V_1=\{v_2, v_7\},V_2=\{v_1, v_3, v_6\},V_3=\{v_4\},V_4= \{v_5\}, V_5=\{v_8\}\}$ is a $SC(G)$-partition of $P_8$. Thus, $SC(P_8)=5$. 

Now, we consider $n=9$. By Theorem \ref{thm:Delta2}, we have $SC(P_9)\leq 6$. We first show that $SC(P_9)\ne 6$. Let $\Omega$ be a $SC(G)$-partition of $P_9$. We consider the following cases.\medskip

\nt {\bf Case 1.} $|V_1|=4$ and $|V_i|=1$ for $2\leq i \leq 6$. Since $\gamma_{st}(P_9)=3$, each singleton set must be a strong coalition of $V_1$, and by Theorem \ref{thm:Delta1}, $V_1$ has
at most three strong coalition partners, a contradiction. Hence, $SC(P_9)\leq 5$.

 \nt {\bf Case 2.} Assume that $|V_1|=3,|V_2|=2$, and $|V_i|=1$ for $3\leq i \leq 6$. Since $\gamma_{st}(P_9)=3$, no two singleton sets in $\Omega$ can form strong coalitions. Since $P_9$ has a unique strong dominating set of cardinality 3, which consist of vertices $v_2$, $v_5$, and $v_8$, then each set containing vertices $v_1$, $v_3$,
$v_4$, $v_6$, $v_7$, and $v_9$ does not appear in any strong coalition of order 3. Without loss of generality, we may assume that $V_2=\{v_2,v_5\}$ and $V_3=\{v_8\}$ such that $V_2 \cup V_3$ is a strong dominating set of order 3. Thus, each of $V_4, V_5$, and $V_6$ must be in a strong coalition with $V_1$, which is impossible, as $V_1 \cup V_4$, $V_1 \cup V_5$, and $V_1 \cup V_6$  have no three vertices in common. So, $SC(P_9)\leq 5$.

 \nt {\bf Case 3.} Let $|V_1|=|V_2|=|V_3|=2$ and $|V_i|=1$ for $4\leq i \leq 6$. Since $\gamma_{st}(P_9)=3$, each
singleton set must be a strong coalition partner of a set of cardinality 2, which is impossible, as $P_9$ has
a unique strong dominating set of cardinality 3.

The partition $\{V_1=\{v_1, v_3, v_5\},V_2=\{v_2, v_4, v_9\},V_3=\{v_6\},V_4= \{v_7\}, V_5=\{v_8\}\}$ is a $SC(G)$-partition of $P_9$. Therefore, $SC(P_9)=5$. 

Now let $n=10$. By Theorem \ref{thm:Delta2}, we have $SC(P_{10})\leq 6$. We shall show that $SC(P_{10})\ne 6$. Let $\Omega$ be a $SC(G)$-partition of $P_{10}$. We consider the following cases.

\nt {\bf Case 1.} $|V_1|=5$ and $|V_i|=1$ for $2\leq i \leq 6$. Since $\gamma_{st}(P_{10})=4$, no two singleton sets in $\Omega$ can produce strong coalitions. Then each singleton set must produce a strong coalition with $V_1$, contradicting Theorem \ref{thm:Delta1}. So, $SC(P_{10})\leq 5$.

\nt {\bf Case 2.} $|V_1|=4$, $|V_2|=2$, and $|V_i|=1$ for $3\leq i \leq 6$. Since $\gamma_{st}(P_{10})=4$, no two singleton sets in $\Omega$ can form strong coalitions. In addition, no singleton set forms a strong coalition with $V_2$. Therefore each of $V_j$ for $2\leq j \leq 6$ must be in a strong coalition with $V_1$, contradicting Theorem \ref{thm:Delta1}. Hence, $SC(P_{10})\leq 5$. 

 \nt {\bf Case 3.} $|V_1|=3,|V_2|=3$, and $|V_i|=1$ for $3\leq i \leq 6$. Every set of $\Omega$ that does not contain either $v_1$ or $v_2$ must be in a strong coalition with a set containing either $v_1$ or $v_2$. Without loss of generality, let $v_1 \in V_3$. Since no two singleton sets in $\Omega$ can form strong coalitions, then each of $V_4$, $V_5$, and $V_6$ cannot be a strong coalition with $V_3$. It follows that $v_1 \in V_1$ or $v_1\in V_2$.
 Due to the vertex $v_1$ can only form a strong coalition of order 4 with vertices $v_3$, $v_6$, and $v_9$, without loss of generality, we may assume that $V_2 \cup V_3=\{v_1,v_3,v_6,v_9\}$.  It can be seen that each of $V_4$, $V_5$, and $V_6$ must be in a strong coalition with $V_1$, which is impossible, as the pairs $V_1 \cup V_4$, $V_1 \cup V_5$, and $V_1 \cup V_6$ that comprise the remaining vertices $v_2,v_4,v_5,v_7,v_8,v_9$, and $v_{10}$ have no three vertices in common. So, $SC(P_{10})\leq 5$. 

\nt {\bf Case 4.} $|V_1|=3,|V_2|=|V_3|=2$, and $|V_i|=1$ for $4\leq i \leq 6$. Since $\gamma_{st}(P_{10})=4$, no two singleton sets can produce a strong coalition.
Furthermore, no singleton set forms a strong coalition with $V_2$ and $V_3$. Now, we consider the following subcases.\medskip

\nt {\bf Subcase (i)}. Let $V_2$ and $V_3$ do not form a strong coalition. Thus each of $V_j$ for $2\leq j \leq 6$ must form a strong coalition with $V_1$, contradicting Theorem \ref{thm:Delta1}. Hence, $SC(P_{10})\leq 5$. 

\nt {\bf Subcase (ii)}. Suppose that $V_2$ and $V_3$ form a strong coalition. Thus, each singleton set must form a strong coalition with $V_1$, which is impossible, as three minimum strong dominating sets have no three vertices in common. So, $SC(P_{10})\leq 5$. 

\nt {\bf Case 5.} $|V_1|=|V_2|=|V_3|=2$ and $|V_i|=1$ for $4\leq i \leq 6$. It can be seen that no union of any two sets has cardinality
4. Consequently, there is no $SC(G)$-partition of order 6. Then, $SC(P_{10})\leq 5$.

The partition $\{V_1=\{v_1, v_3, v_6\},V_2=\{v_2, v_7, v_8, v_{10}\}, V_3=\{v_4\}, V_4= \{v_5\}, V_5=\{v_9\}\}$ is a $SC(G)$-partition of $P_{10}$. Thus, $SC(P_{10})=5$. 
    
Next assume that $n=11$. By Theorem \ref{thm:Delta2}, we have $SC(P_{11})\leq 6$. Now, we show that $SC(P_{11})\ne 6$. Let $\Omega$ be a $SC(G)$-partition of $P_{11}$. We consider the following cases.

\nt {\bf Case 1.}$|V_1|=6$ and $|V_i|=1$ for $2\leq i \leq 6$. Since $\gamma_{st}(P_{11})=4$, no two singleton sets in $\Omega$ can form strong coalitions. Then each singleton set must produce a strong coalition with $V_1$, which contradicts Theorem \ref{thm:Delta1}. Hence, $SC(P_{11})\leq 5$.

\nt {\bf Case 2.}$|V_1|=5$, $|V_2|=2$, and $|V_i|=1$ for $3\leq i \leq 6$. Since $\gamma_{st}(P_{11})=4$, each of $|V_j|=1$ for $2\leq j \leq 6$ must be in a strong coalition with $V_1$, and by Theorem \ref{thm:Delta1}, $V_1$ is in
at most three strong coalition partners, a contradiction. So, $SC(P_{11})\leq 5$.

 \nt {\bf Case 3.} $|V_1|=4,|V_2|=3$, and $|V_i|=1$ for $3\leq i \leq 6$. Since $\gamma_{st}(P_{11})=4$, each singleton set must produce a strong coalition with $V_1$ or $V_2$. It can be seen that $P_{11}$
  possesses only three strong coalitions of order 4 such as $A=\{v_2,v_5,v_8,v_{10}\}$, $B=\{v_2,v_5,v_7,v_{10}\}$, and $C=\{v_2,v_4,v_7,v_{10}\}$. Now, we consider the following subcases.\medskip

\nt {\bf Subcase (i)}. Without loss of generality, assume that $V_2  \cup V_3=\{v_2,v_5,v_7,v_{10}\}$ is a strong coalition of order 4. Note that each of $V_4$, $V_5$, and $V_6$ is in a strong coalition with $V_1$. It holds that $\bigcap_{k=4}^{6}|V_1 \cup V_k|=4$, which is impossible, as the pairs $V_1 \cup V_4$, $V_1 \cup V_5$, and $V_1 \cup V_6$ comprising the remaining vertices $v_1,v_3,v_4,v_6,v_8$, $v_9$ and $v_{11}$ have no four vertices in common. So, $SC(P_{11})\leq 5$. 

\nt {\bf Subcase (ii)}. Let $V_2 \cup V_3=\{v_2,v_5,v_8,v_{10}\}$.  Assume that each of $V_4$, $V_5$, $V_6$ can be a strong coalition with $V_1$. Then $\bigcap_{k=4}^{6}|V_1 \cup V_k|=4$, which is impossible, as the pairs $V_1 \cup V_4$, $V_1 \cup V_5$, and $V_1 \cup V_6$ containing the remaining vertices have no four vertices in common. Hence, $SC(P_{11})\leq 5$.

\nt {\bf Subcase (iii)}. Let $V_2 \cup V_3=\{v_2,v_4,v_7,v_{10}\}$.  Suppose that each of $V_4$, $V_5$, $V_6$ can be a strong coalition with $V_1$. It follows that $\bigcap_{k=4}^{6}|V_1 \cup V_k|=4$, which is impossible, as the pairs $V_1 \cup V_4$, $V_1 \cup V_5$, and $V_1 \cup V_6$ containing the remaining vertices have no four vertices in common. Thus, $SC(P_{11})\leq 5$.

\nt {\bf Subcase (iv)}. Let $V_2=\{v_2,v_5,v_{10}\}$, $V_3=\{v_7\}$, and $V_4=\{v_8\}$ such that $V_2 \cup V_3$ and $V_2 \cup V_4$ are strong coalitions of order 4 of $P_{11}$. In addition, each of $V_5$ and $V_6$ can be in a strong coalition with $V_1$. Then $\bigcap_{k=5}^{6}|V_1 \cup V_k|=4$, which is impossible, as the pairs $V_1 \cup V_5$ and $V_1 \cup V_6$ comprising the remaining vertices $v_1,v_3,v_4,v_6,v_9$, and $v_{11}$ have no four vertices in common. So, $SC(P_{11})\leq 5$.

\nt {\bf Subcase (v)}. Assume that $V_2=\{v_2,v_7,v_{10}\}$, $V_3=\{v_4\}$, and $V_4=\{v_5\}$ such that $V_2 \cup V_3$ and $V_2 \cup V_4$ are strong coalitions of order 4 of $P_{11}$. As before, each of $V_5$ and $V_6$ can be in a strong coalition with $V_1$. It holds that $\bigcap_{k=5}^{6}|V_1 \cup V_k|=4$, which is impossible, as the pairs $V_1 \cup V_5$ and $V_1 \cup V_6$ containing the remaining vertices have no four vertices in common. Thus, $SC(P_{11})\leq 5$.

\nt {\bf Subcase (vi)}. Without loss of generality, assume that each of $V_3$, $V_4$, and $V_5$ form a strong coalition with $V_2$. It is clear that $V_6$ can form a strong coalition with $V_1$. Then $\bigcap_{k=3}^{5}|V_2 \cup V_k|=3$, which is impossible, as there are two sets such as $A$ and $C$ such that $|A \cap C|=2$.

\nt {\bf Case 4.} $|V_1|=4,|V_2|=|V_3|=2$, and $|V_i|=1$ for $4\leq i \leq 6$. Since $\gamma_{st}(P_{11})=4$, no two singleton sets can produce a strong coalition.
Furthermore, no singleton set forms a strong coalition with $V_2$ and $V_3$. So each of $V_j$ for $2\leq j \leq 6$ must be in a strong coalition with $V_1$, which contradicts Theorem \ref{thm:Delta1}. Hence, $SC(P_{11})\leq 5$.

\nt {\bf Case 5.} $|V_1|=|V_2|=3$, $|V_3|=2$, and $|V_i|=1$ for $4\leq i \leq 6$. From our previous discussions, each of $V_j$ for $3\leq j \leq 6$ must produce a strong coalition with $V_1$ or $V_2$. Since the graph $P_{11}$
 has only two strong coalitions of order 4, a singleton set cannot form a strong coalition with $V_1$ or $V_2$, a contradiction. Consequently, there is no $SC(G)$-partition of order 6. Then, $SC(P_{11})\leq 5$.

\nt {\bf Case 6.} $|V_1|=3$, $|V_i|=2$ for $2\leq i \leq 4$, and $|V_5|=|V_6|=1$. Since $\gamma_{st}(P_{11})=4$, as before, each of $V_j$ for $2\leq j \leq 6$ must be in a strong coalition with $V_1$. By Theorem \ref{thm:Delta1}, $V_1$ has 
at most three strong coalition partners, a contradiction. Hence, $SC(P_{11})\leq 5$. 

\nt {\bf Case 7.} $|V_i|=2$ for $1\leq i\leq 5$ and $|V_6|=1$. Then, $V_6$ together with any other set cannot have four vertices and hence cannot form a strong coalition. Consequently, there is no $SC(G)$-partition of order 6. Then, $SC(P_{11})\leq 5$.

The partition $\{V_1=\{v_1, v_3, v_6, v_9\},V_2=\{v_2, v_5, v_{10}\}, V_3=\{v_7\}, V_4=\{v_8\}, V_5=\{v_{11}\}\}$ is a $SC(G)$-partition of $P_{11}$. Thus, $SC(P_{11})=5$. 

Finally, let $n\geq 12$. By Theorem \ref{thm:Delta2}, for any path $P_n$ we have $SC(P_{n})\leq 6$. Now, we find a maximal strong coalition partition of order 6 for $P_n$, where $n\geq 12$. Now, We consider the following cases.
\begin{itemize}
\item[(i)]
$n=2k$ for $k\geq 6$. In this case, let 

$\Omega_1=\{v_1,v_3,v_6,v_{11},v_{13},v_{15},\ldots,v_{2k-1}\}$,
$\Omega_2=\{v_2,v_7,v_{10},v_{12},v_{14},v_{16},\ldots,v_{2k}\}$, $\Omega_3=\{v_4\}$, $\Omega_4=\{v_5\}$, $\Omega_5=\{v_8\}$ and $\Omega_6=\{v_9\}$.
Now, $\Omega_1$ forms strong coalitions with $\Omega_5$ and $\Omega_6$, and $\Omega_2$ forms strong coalitions with $\Omega_3$ and $\Omega_4$.
\item[(ii)]
$n=2k+1$ for $k\geq 6$. In this case, let 

$\Omega'_1=\{v_1,v_3,v_6,v_{11},v_{13},v_{15},\ldots,v_{2k+1}\}$,
$\Omega'_2=\{v_2,v_7,v_{10},v_{12},v_{14},v_{16},\ldots,v_{2k}\}$, $\Omega'_3=\{v_4\}$, $\Omega'_4=\{v_5\}$, $\Omega'_5=\{v_8\}$ and $\Omega'_6=\{v_9\}$.
Now, $\Omega'_1$ forms strong coalitions with $\Omega'_5$ and $\Omega'_6$, and $\Omega'_2$ forms strong coalitions with $\Omega'_3$ and $\Omega'_4$.
\end{itemize}
Therefore the proof is complete.
\qed
\end{proof}

\section{Conclusion}
This paper introduces the concept of strong coalition in graphs. We have investigated the existence of strong coalition partitions and presented some lower bound and upper bounds for the strong coalition number. We presented a sufficient condition for graphs $G$ of order at least $4$ with $SC(G)=|V(G)|$. Also, we have determined the precise values of ${SC}(P_n)$, ${SC}(K_{r,s})$, and ${SC}(C_n)$. We have outlined some unresolved problems and potential research directions related to the strong  coalition number of graphs. However, there is still much work to be done in this area.
\begin{enumerate}
	\item What is the strong  coalition number of  graph operations, such as corona, Cartesian product, join, lexicographic, and so on?
	
	\item Which graphs have small and large strong coalition number?
	% (see e.g. \cite{BIMS})?
	
	\item What is the effects on ${SC}(G)$ when $G$ is modified by operations on vertex and edge of $G$?

	\item Study Nordhaus and Gaddum lower and upper bounds on the sum and the product
	of the strong coalition number of a graph and its complement.

	\item   Similar to the coalition graph of $G$, it is natural to define and study the strong coalition graph of $G$ for strong coalition partition $\Omega$, which can be denoted by $SCG(G, \Omega)$, and is defined as follows. Corresponding to any strong coalition partition $\Omega=\{V_1,V_2,\ldots, V_k\}$ in a graph $G$, a  {\em strong coalition graph} $SCG(G, \Omega)$ is  associated in which there is a one-to-one correspondence between the  vertices of $SCG(G, \Omega)$  and the sets $V_1, V_2,...,V_k$ of $\Omega$,  and two vertices of $SCG(G, \Omega)$  are adjacent if and only if their corresponding sets in $\Omega$ form a strong coalition.

\end{enumerate}

\medskip

\noindent{\bf Acknowledgement.} 
The work of Hamidreza Golmohammadi was supported by the
Mathematical Center in Akademgorodok, under agreement No. 075-15-2022-281 with
the Ministry of Science and High Education of the Russian Federation.

\end{document}